\documentclass[11pt]{amsart}

\usepackage{amsmath,amsthm,amsfonts,amssymb,amscd,enumerate}
\theoremstyle{plain} \numberwithin{equation}{section}

\usepackage{tikz}
\usepackage{tkz-graph}

\newtheorem{theorem}{Theorem}
\newtheorem{corollary}[theorem]{Corollary}
\newtheorem{conjecture}{Conjecture}
\newtheorem{lemma}[theorem]{Lemma}
\newtheorem{proposition}[theorem]{Proposition}
\newtheorem{claim}{Claim}
\theoremstyle{definition}
\newtheorem{dfn}{Definition}

\newtheorem{remark}[theorem]{Remark}
\newtheorem{example}{Example}

\def\PL{\mathcal{P}(L)}
\def\SL{\mathcal{S}(L)}
\def\S{\mathcal{S}}
\def\SL{\mathcal{S}(L)}

\def\d{\partial}
\def\P{\mathcal{P}}

\def\G{\textbf{G}}
\def\MA{M(\mathcal{A})}
\def\SA{S(\mathcal{A})}
\def\M{\bar{M}(\mathcal{A})}
\def\BS{\bar{S}(\mathcal{A})}
\def\A{\mathcal{A}}

\begin{document}

\title{The Action dimension of Artin Groups}


\author{Giang Le}
\address{Math department, Ohio State University, 231 West 18th Avenue, Columbus OH 43210}
\email{le.145@osu.edu}
\thanks{This research was partially supported by the NSF, under grant number DMS-1510640}


\date{}

\begin{abstract}
The \emph{action dimension} of a discrete group $G$ is the minimum dimension of a contractible manifold, which admits a proper $G$-action. In this paper, we study the action dimension of general Artin groups. The main result is that the action dimension of an Artin group with the nerve $L$ of dimension $n$ for $n \ne 2$ is less than or equal to $(2n + 1)$ if the Artin group satisfies the $K(\pi, 1)$-Conjecture and the top cohomology group of $L$ with $\mathbb{Z}$-coefficients is trivial. For $n = 2$, we need one more condition on $L$ to get the same inequality; that is the fundamental group of $L$ is generated by $r$ elements where $r$ is the rank of $H_1(L, \mathbb{Z})$.
\end{abstract}

\maketitle

\section{Introduction}\label{section:introduction}
The action dimension was first defined by M. Bestvina, M. Kapovich and B. Keiner in their paper titled \emph{Van Kampen's embedding obstruction for discrete groups} \cite{BKK}. In that paper, the authors introduced the so-called obstructor dimension of a discrete group and proved a theorem which states that if a group has obstructor dimension $m$ then the group cannot act properly on a contractible manifold of dimension less than $m$. From this result, the action dimension, denoted by \emph{actdim}, which is an upper bound for the obstructor dimension, is defined. In the same paper, as an example, the obstructor dimensions of the braid groups were calculated, which in turn can be used to determine their action dimensions. Another paper, in which the action dimension was studied, is \emph{The action dimension of right-angled Artin groups} by G. Avramidi, M. Davis, B. Okun and K. Schreve \cite{ADOS}. In their paper, the authors studied right-angled Artin groups (RAAGs) and determined the case when a RAAG has the maximum action dimension. 

Braid groups or RAAGs are examples of Artin groups, which were introduced by J. Tits in 1965 as extensions of Coxeter groups. Braid groups and RAAGs are well understood because of their additional geometric properties; however,  Artin groups in general are poorly understood. Many classical problems in group theory for Artin groups are still open, for example, questions about torsion-freeness, the word problem, the conjugation problem,... In this paper, we study the action dimension of a class of Artin groups which satisfy the so-called $K(\pi, 1)$-Conjecture. 

To set up the main theorem, we need to introduce some notation. An Artin group is a group with a presentation by a system of generators $a_i, i \in I$, and relations
$$a_ia_ja_i \dots = a_j a_i a_j \dots ; \text{   }i, j \in I $$
where the words on each side of these relations are sequences of $m_{ij}$ $(= m_{ji})$ alternating letters $a_i$ and $a_j$. We set $m_{ij} = \infty$ to mean that there is no relation between $a_i$ and $a_j$, and $m_{ii} = 1$.
The symmetric matrix $M = (m_{ij})_{i,j \in I}$ is a \emph{Coxeter matrix} on $I$. Corresponding to a Coxeter matrix is a \emph{Coxeter group}, which has the same definition as the Artin group but with additional relations $a_i^2 = 1$, for all $i \in I$. A Coxeter group and an Artin group can also be determined by a graph $\Gamma$, called the \emph{Coxeter diagram}, which has $I$ as its set of vertices and any two vertices $i$ and $j$ are connected if and only if $m_{ij} \ge 3$. If $\Gamma$ is connected, we say that the Artin group (and the Coxeter group) is \emph{irreducible}.

Related to an Artin group (or a Coxeter group) is a simplicial complex, called the \emph{nerve} $L$ of the Artin group (or the Coxeter group). The definition of the nerve is as follows: the vertices of $L$ are indexed by $I$, the index set of generators of the Artin group. Any collection of $k$ vertices of $L$ forms a $(k-1)$-simplex if and only if the Coxeter subgroup generated by those  corresponding $k$ generators is finite. We usually denote the Artin group and the Coxeter group corresponding to the nerve $L$ by $A_L$ and $W_L$, respectively. When the nerve $L$ is a simplex $\sigma$, we call it a \emph{spherical} Artin group $A_\sigma$ and the Coxeter group in this case is finite.

Let $L$ be a simplicial complex with edges labeled by integers greater than or equal to 2, and $A_L$ be the Artin group associated with $L$. The main result of this paper is a manifold construction which proves the following theorem:
 
\begin{theorem}\label{theorem:maintheorem}
If $A_L$ satisfies the $K(\pi, 1)$-Conjecture and $H^n(L, \mathbb{Z}) = 0$ (and when $n = 2$, further suppose that $\pi_1(L)$ is generated by $r$ elements where $r = rk H_1(L, \mathbb{Z})$), then actdim$(A_L) \le 2n + 1$.
\end{theorem}
Charney and Davis \cite{CharneyDavis1} proved that when $L$ is a flag complex, then $A_L$ satisfies the $K(\pi, 1)$-Conjecture. Thus, the main theorem implies the following corollary.

\begin{corollary}\label{cor:flagcomplex}
If $L$ is an $n$-dimensional flag complex of dimension $n \ne 2$ and the top cohomology group $H^n(L, \mathbb{Z})$ is trivial, then the action dimension of the Artin group $A_L$ is less than or equal to $2n + 1$.
\end{corollary}

In Section \ref{section:Spherical_Artin_Groups}, the action dimension of irreducible spherical Artin groups is studied.  In this case, the action dimension of the irreducible Artin group $A_\sigma$ is exactly $2n + 1$. 

\begin{corollary}\label{cor:irreducible}
If the Artin group $A_L$ in Theorem \ref{theorem:maintheorem} has an irreducible spherical Artin subgroup with the nerve of the same dimension, then actdim($A_L$) $= 2n + 1$.
\end{corollary}

\medskip
We prove the upper bound on the action dimension of the Artin group by constructing an aspherical manifold which has dimension $2n + 1$. The idea of our construction is that we consider the Artin group as the colimit group (and/or the fundamental group) of a poset of groups over the poset of simplices of $L$ or a poset whose geometric realization is the cone on the barycentric subdivision of $L$. A poset of groups has realization by aspherical complexes, which is a complex whose fundamental group is the fundamental group of the poset of groups. We will use manifolds instead of complexes in the realization to construct the manifold.

\medskip
There are two points in our proof of the main result. Firstly, we need to construct an aspherical manifold, whose fundamental group is the Artin group (thus, the Artin group acts on its universal cover, which is a contractible manifold). For this we will need the condition that the Artin group satisfies the $K(\pi, 1)$-Conjecture. Charney and Davis \cite{CharneyDavis} proved that the $K(\pi, 1)$-Conjecture is equivalent to a conjecture regarding the contractibility of the basic construction of a poset of groups for the Artin group. Haefliger \cite{Haefliger} related, under certain conditions on the poset of groups, the contractibility of the basic construction and the sphericality of a realization by aspherical complexes of the poset of groups. We will use these two results to prove that our manifold is aspherical. Secondly, we want to minimize the dimension of our manifold. The best dimension we can have is $2n + 1$ because the dimensions of the manifolds for local groups are already $2n + 1$. We can always construct an aspherical manifold of dimension $2n + 2$ for the Artin group; however, if we want to decrease the dimension by one, we need to put some condition on the nerve $L$. In particular, the constraint on the cohomology group of $L$ implies that $L$ can be embedded into a contractible complex of the same dimension. This is a crucial point in our proof.

This paper is divided into the following sections: Section \ref{section:Poset_of_groups} deals with poset of groups and poset of spaces, most of this follows the work in \cite{BH}. We also prove some technical lemmas which will be needed later. Section \ref{section:Spherical_Artin_Groups} provides basic information about Artin groups and different dimensions of a group (geometric and action dimensions). We study the geometric and action dimensions of irreducible spherical Artin groups here. In Section \ref{section:Manifold_construction} we describe our manifold construction and prove our results. 


\section{Posets of groups and posets of spaces}\label{section:Poset_of_groups}
In this section, we provide some basic information about posets of groups (or simple complexes of groups) and describe an important construction which is useful in building many examples of group actions on complexes. The idea is that if an action of a group $G$ by isometries on a complex $X$ has a strict fundamental domain  $Y$ (i.e. a subcomplex of $X$ that meets each orbit in exactly one point), then one can recover $X$ and the action of $G$ directly from $Y$ and the pattern of its isotropy subgroups.

\subsection{Posets.} A \emph{poset} is a partially ordered set. Associated to any poset $\P$ one has a simplicial complex whose set of vertices is $\P$ and whose $k$-simplices are the strictly increasing sequences $\sigma_0 < \sigma_1 < \ldots < \sigma_k$ of elements of $\P$. The geometric realization of this simplicial complex will be called the \emph{geometric realization of the poset} $\P$, denoted $|\P|$. 
\begin{example}
Let $\P = \PL$ the poset of a simplicial complex $L$. $\P$ is partially ordered by inclusion. Then the simplicial complex associated to $\P$ is the barycentric subdivision of $L$.

Let $\mathcal{S} = \SL = \PL \cup \{ \emptyset \}$. The simplicial complex associated to $\S$ is the cone on the barycentric subdivision of $L$.
\end{example}

\subsection{Dual cones}
Let $\P$ be a poset. The \emph{dual cone} of $\sigma \in \P$, denoted by $D_\sigma$, is the geometric realization of the set $\P_{\ge \sigma}$ of all elements $\tau \in \P$ such that $\tau \ge \sigma$.
 
If $\P$ is the poset of non-empty simplexes of a simplicial complex $L$, then the dual cone of a simplex $\sigma \in L$ is the cone on the barycentric subdivision of the link of $\sigma$ in $L$, $Lk'_\sigma$. It is also a subcomplex of $L'$, the barycentric subdivision of $L$. Some examples of the dual cones of simplex $\sigma \in L$ are:
\begin{enumerate}
\item If $\sigma$ is top-dimensional, then $D_\sigma$ is a point.
\item If the codimension of $\sigma$ is 1, then $D_\sigma$ is a cone on a discrete set.
\item If the codimension of $\sigma$ is 2, then the link of $\sigma$ is a graph and $D_\sigma$ is the cone on the barycentric subdivision of this graph.
\item if $L$ is a PL-triangulation of a manifold then $D_\sigma$ is the dual cell of $\sigma$.
\end{enumerate}

\subsection{Stratified spaces}
This subsection contains background about stratified space, a useful generalization of simplicial complex. The discussion follows from \cite{BH}, chapter II.12.
\begin{dfn}
(Stratified sets and spaces). A \emph{stratified set} $(X, \{X^\sigma\}_{\sigma \in \P})$ consists of a set $X$ and a collection of subsets $X^\sigma$ called \emph{strata}, indexed by a set $\P$, such that
\begin{enumerate}
\item $X$ is a union of strata,
\item if $X^\sigma = X^\tau$ then $\sigma = \tau$,
\item if an intersection $X^\sigma \cap X^\tau$ of two strata is non-empty, then it is a union of strata, 
\item for each $x \in X$ there is a unique $\sigma(x) \in \P$ such that the intersection of the strata containing $x$ is $X^{\sigma(x)}$.
\end{enumerate}
\end{dfn}

The inclusion of strata gives a partial ordering on the set $\P$, namely $\tau \le \sigma$ if and only if $X^\tau \subseteq X^\sigma$. We shall often refer to $(X, \{X^\sigma\}_{\sigma \in \P})$ as "a stratified space $X$ with strata indexed by the poset $\P$" or "a stratified space over $\P$".

\begin{example}
The geometric realization $|\P|$ of a poset $\P$ has a natural stratification indexed by $\P$: the stratum $|\P_{\le \sigma}|$ indexed by $\sigma \in \P$ is the union of the $k$-simplices $\sigma_0 < \sigma_1 < \ldots < \sigma_k$ with $\sigma_k \le \sigma$.
\end{example}

\subsection{Posets of groups}
Let $\P$ be a poset. A {\em poset of groups} over a poset $\P$ is a functor $\G (\P)$ from $\P$ to the category of groups. More precisely, we have the following definition.
\begin{dfn} (Poset of Groups).
A {\em poset of groups} $\G(\P) = (G_\sigma, \psi_{\sigma \tau})$ over a poset $\P$ consists of the following data:
\begin{enumerate}
\item for each $\sigma \in \P$, a group $G_\sigma$, called the {\em local group} at $\sigma$;
\item for each $\sigma < \tau$, an injective homomorphism $\psi_{\sigma \tau} : G_\sigma \to G_\tau$ such that if $\rho < \sigma < \tau$, then $$\psi_{\rho \tau} = \psi_{\rho \sigma} \psi_{\sigma \tau}$$
\end{enumerate}
\end{dfn}

We sometimes denote the homomorphism $\psi_{\sigma \tau}$ by $\psi_a$, where $a$ stands for the morphism $\sigma < \tau$, we call $a$ an \emph{edge} and write the initial vertex: $i(a) = \sigma$ and the terminal vertex: $t(a) = \tau$.

The notion of poset of groups is similar to the notion of simple complex of groups, which is explained in \cite{BH}. Complexes of groups arise naturally from an action of a group $G$ by isometries on a complex $X$, which has a strict fundamental domain $Y$. There are local groups $G_\sigma$ for each simplex $\sigma < Y$, which are the isotropy subgroups of $G$. And whenever one cell $\sigma$ is contained in another $\tau$, we have injective map $\psi_{\tau \sigma} : G_\tau \to G_\sigma$. There are inclusion maps $\phi_\sigma: G_\sigma \to G$ which are compatible with $\psi_{\tau \sigma}$: $\phi_\sigma \psi_{\tau \sigma} = \phi_\tau$ whenever $\sigma \subset \tau$. In the notion of complex of group, if $\sigma$ is a face of $\tau$ in $Y$, then $G_\tau$ is a subgroup of $G_\sigma$. For the purpose of studying Artin groups, we want to reverse the order relation so that if $\sigma$ is a face of $\tau$, then $G_\sigma$ is a subgroup of $G_\tau$. Thus, we will use the notion of poset of groups instead of simple complex of groups as developed in \cite{BH}. 

\subsection{Classifying space $BG(\P)$ of a poset of groups}
We associate to a poset of groups $\G(\P)$ over $\P$ the small category $CG(\P)$ whose objects are the objects of $\P$ and whose morphisms are the pairs $(g, \alpha)$, where $\alpha$ is in the set of morphisms of the poset $\P$ and $g \in G_{t(\alpha)}$. We define maps $i, t : CG(\P) \to V(\P)$ by $i((g,\alpha)) = i(\alpha)$ and $t((g, \alpha)) = t(\alpha)$. The composition $(g,\alpha)(h,\beta)$ is defined if $i(\alpha) = t(\beta)$ and then it is equal to 
$$(g, \alpha)(h, \beta) = (g \psi_\alpha(h), \alpha \beta).$$
The condition on $\psi_\alpha$ in the definition of a complex of groups implies the associativity of this law of composition. The map $(g, \alpha) \to \alpha$ is a functor $p$ of $CG(\P)$ on $\P$. 

\textbf{Classifying space $BG(\P)$ of $\G(\P)$}. Recall that the nerve $N(C)$ of a small category $C$ is a simplicial complex constructed from the objects and morphisms of $C$. There is a 0-simplex of $N(C)$ for each object of $C$. There is a 1-simplex for each morphism $f: x \to y$ in $C$. In general, the set $N(C)_k$ of $k$-simplexes of the nerve consists of the $k$-tuples of composable morphisms of $C$
$$x_0 \to x_1 \to x_2 \ldots \to x_{k-1} \to x_k.$$

 The geometric realization of the nerve of the category $CG(\P)$ is denoted by $BG(\P)$ and is called the \emph{classifying space} of $G(\P)$.
\begin{example} If $\P$ has only one object $\sigma$, then $BG(\P)$ is the usual classifying space $BG_\sigma$ of the group $G_\sigma$.
\end{example}

\subsection{Groups associated to a poset of groups}\label{section:groups_cxofgroups}
Associated to any poset of groups $\G(\P)$, there are two groups: one is the colimit of the system of groups and monomorphisms $(G_\sigma, \psi_{\sigma \tau})$:
$$G := \varinjlim_{\sigma \in \P} G_\sigma $$
The group $G$ is obtained by taking the free product of the groups $G_\sigma$ and making the identifications $\psi_{\sigma \tau} (h) = h$, $\forall h \in G_\sigma$, $\forall \sigma < \tau$.

The other group is the \emph{fundamental group} $\pi_1 := \pi_1(G(\P), \sigma)$ of a poset of groups $G(\P)$ on a poset $\P$ based at a vertex $\sigma$ of $|\P|$, which is the fundamental group of the geometric realization $BG(\P)$ of the nerve of $CG(\P)$ based at $\sigma$.
A presentation of the fundamental group can be found explicitly \cite{BH}. For the purpose of this paper, we only care about the case when these two groups are the same. And it happens when $|\P|$ is simply connected \cite{Haefliger}.[Reference, Haefliger].

\subsection{The basic construction}
Let $Q$ be a group. A {\em simple morphism} $\varphi = (\varphi_\sigma)$ from $G(\P)$ to $Q$, written $\varphi : G(\P) \to Q$, is a map that associates to each $\sigma \in \P$ a homomorphism $\varphi_\sigma : G_\sigma \to Q$ such that if $ \sigma < \tau$, then $\varphi_\sigma = \varphi_\tau \psi_{\sigma \tau}$. We say that $\varphi$ is {\em injective on the local groups} if $\varphi_\sigma$ is injective for each $\sigma \in \P$.

The natural homomorphisms $\iota_\sigma : G_\sigma \to G$ from local groups to the colimit group of the poset of groups give a canonical simple morphism $\iota : G(\P) \to G$, where $\iota = (\iota_\sigma)$. In general, $\iota_\sigma$ is not injective.

The complex of groups is {\em developable} if the natural map $G_\sigma \to G$ is injective for each $\sigma \in \P$.

Suppose we have a poset of groups $\G(\P)$ over a poset $\P$, a stratified space $(Y, \{Y^\sigma\}_{\sigma \in \P} )$ indexed by the poset $\P$ and injective simple morphism $\varphi : G(\P) \to Q$ for some group $Q$, then we can build a stratified space on which the group $Q$ acts on and with strict fundamental domain $Y$. 

Since the simple morphism $\varphi : G(\P) \to Q$ is injective on the local groups, we can identify each local group $G_\sigma$ with its image $\varphi(G_\sigma)$ in $Q$. This basic construction can be defined as follows. Let
$$U(Q, Y) = Q \times Y / \sim,$$
the equivalence $\sim$ is defined by
$$(g, y) \sim (g', y') \iff y = y' \text{and } g^{-1}g' \in G_{\sigma(y)},$$
where $X^{\sigma(y)}$ is the smallest stratum containing $y$.

We write $[g, y]$ to denote the equivalence class of $(g, y)$. The group $Q$ acts by strata preserving automorphisms according to the rule
$$g' . [g, y] = [g'g, y].$$
And if we identify $Y$ with the image of $\{1\} \times Y$ in $U(Q, Y)$ (where 1 is the identity element of $Q$), then $Y$ is a strict fundamental domain for the action and the associated complex of groups is $G(\P)$.

The simplest case of the basic construction is when we take $Q = G$ the colimit group and $Y =  |\P|$. The construction gives us a complex $U(G, |\P|)$ with an action of $G$ such that the quotient $U(G, |\P|)/G$ is $|\P|$. The universal cover of the poset of group $U$ can be described as the cover of $U(G, |\P|)$ corresponding to the kernel of the map
$$\pi_1 \to G$$
 Next, we will study the basic construction in various situations. A group $G$ can be the colimit of other systems of groups formed from the system $\{ \G_\sigma \}$ in the two following settings:
\begin{enumerate}
\item Let $\P = \PL$ is the poset of (nonempty) cells of a simplicial complex $L$ and $\S = \SL = \PL \cup \{ \emptyset \}$. Then the geometric realization $|\S|$ is the cone on $|\P|$. We can form a complex of groups $\G(\S)$ from $\G(\P)$ by defining $G_{\emptyset} := \{ 1\}$ and for other $\sigma \in \S$ the local group at $\sigma$ of $\G(\S)$ is the same as the local group at $\sigma$ of $\G(\P)$.

\item Suppose that $L$ is a full subcomplex of $L'$. Define a new simple complex of groups $\G(\P(L'))$ over $\P(L')$ with groups $\{ G'_\sigma\}_{\sigma \in \P(L')}$ by
\[ G'_\sigma =
	\begin{cases}
		G_{\sigma \cap L} & \quad \text{if } \sigma \cap L \ne \emptyset\\
		\{e\} &\quad \text{if } \sigma \cap L = \emptyset\\
	\end{cases}
\]
It is easy to see that the system of groups in $\G(\P(L'))$ has the same colimit as $\G(\PL)$. 
\end{enumerate}
We will consider the basic construction with the same group $G$ and with different stratified spaces.

\begin{lemma}\label{le:colimit=fundgroup}
Suppose $\G(\P)$ is a developable poset of groups over a poset $\P$. Let $G$ be the colimit of the $G_\sigma$, $\sigma \in \P$. Assume the geometric realization of $\P$ is simply connected. Then
\begin{enumerate}
\item The fundamental group of $\G(\P)$ is $G$.
\item The universal cover of $\G(\P)$ is $U(G, |\P|)$.
\end{enumerate}
\end{lemma}
\begin{proof}
If the geometric realization of $\P$ is simply connected, then as we stated in Section \ref{section:groups_cxofgroups}, the fundamental group $\pi_1$ of the poset of group is the same as the colimit $G$ of the system $\{ G_\sigma \}$. Since the two groups are the same, $U(G, |\P|)$ is the universal cover.
\end{proof}

\begin{lemma}\label{le:hom.equiv}
Suppose $L$ is a contractible simplicial complex. Then $U(G, |\PL|)$ is $G$-equivariantly homotopy equivalent to $U(G, |\SL|)$.
\end{lemma}

\begin{proof}
The geometric realization of $\PL$ is the barycentric subdivision of $L$. Thus, it is contractible. $|\SL|$ is the cone on $|\PL|$ and $U(G, |\SL|)$ is formed from $U(G, |\PL|)$ by coning off each copy of $|\PL|$. It follows that $U(G, |\PL|)$ and $U(G, |\SL|)$ are homotopy equivalent.
\end{proof}

Let $L$ and $L_1$ be simplicial complexes and $L$ is a full subcomplex of $L_1$. Let $\G(\PL)$ be a poset of groups on the poset $\PL$. Define new posets of groups $\G(\SL)$, $\G(\P(L_1))$ and $\G(\S(L_1))$ as above.
\begin{lemma}\label{lemma:UPandUS}
Suppose $L < L_1$ is a full subcomplex and $L'$ is contractible. Then $U(G,|\SL|)$ is $G$-equivariantly homotopy equivalent to $U(G, |\P(L_1)|)$.
\end{lemma}

\begin{proof}
There is a deformation retraction $r: \S(L_1) \to \SL$ which sends each vertex of $L_1 - L$ to the cone point corresponding to $\emptyset \in \SL$. Recall that for each $x \in L$, we define $\sigma(x)$ to be the minimum simplex $\sigma_0$ such that $x$ belongs to a unique (open) simplex corresponding to a chain $\sigma_0 < \sigma_1 < \ldots < \sigma_k$ in $\P$. For any $x \in \S(L_1)$, we have $G'_{\sigma(x)} = G_{\sigma(r(x))}$. It follows that $r$ induces a $G$-equivariant deformation retraction $U(G, |\S(L_1)|) \to U(G, |\S(L)|)$. By Lemma \ref{le:hom.equiv}, $U(G, |\S(L_1)|)$ is $G$-equivariantly homotopy equivalent to $U(G, |\P(L_1)|)$. The lemma follows.
\end{proof}

The following proposition is corollary 3.2.4 from \cite{Haefliger}.
\begin{proposition}\label{prop:contractibility}
Suppose $G(\P)$ is a developable poset of groups, and let $G$ be the colimit of the $G_\sigma$. Suppose $|\P|$ is simply connected. Then $BG(\P)$ is homotopy equivalent to the Eilenberg-MacLane space $BG$ if and only if $U(G,|\P|)$ is contractible.
\end{proposition}

Since $|\P|$ is simply connected, the fundamental group of $BG(\P)$ is $G$. Let $EG$ be the universal cover of $BG$ and let $U = U(G, |\P|)$ be the basic construction. Haefliger proved that $BG(\P)$ is homotopy equivalent to the Borel construction, $U \times_G EG$. This proves the proposition.

\subsection{Posets of spaces} \label{subsec:ComplexofSpaces}
Associate with a poset of groups is a poset of spaces. Given a simplicial complex and a collection of spaces, we can construct a simplicial complex which gives rise to a poset of groups.

\textbf{Realizations with aspherical complexes.} Let $\P$ be a poset. For each $\sigma \in \P$, define $P_\sigma := |\P_{\le \sigma}|$ and call it the \emph{cone} of $\sigma$. There is a natural inclusion $i_\sigma: P_\sigma \to |\P|$ and If $a = (\sigma <\tau)$ is an edge from $\sigma$ to $\tau$, then there is a natural inclusion $i_a : P_\sigma \to P_\tau$.
\begin{remark}
If $\P = \PL$ is the poset of cells in a cell complex $L$, then $P_\sigma$ is the cell corresponding to the barycentric subdivision of $\sigma$. Similarly, if $\P = \SL$, then $P_\sigma$ is the cone on the barycentric subdivision of $\sigma$. Also, $i_a : P_\sigma \to P_\tau$ is the inclusion of one cell as a face of the other.
\end{remark}

Suppose $Y$ is a space and $\pi : Y \to |\P|$ is a continuous map. For each $\sigma \in \P$, denote $Y(\sigma) = \pi^{-1}(\sigma)$. Let $Y(P_\sigma)$ be the subspace of $Y \times P_\sigma$ of all $(y, x)$ such that $\pi(y) = i_\sigma(x)$. In other words, $Y(P_\sigma)$ is the graph of the restriction of $\pi$ to $\pi^{-1}(P_\sigma)$. There are two projection maps: 
$$\pi_\sigma : Y(P_\sigma) \to P_\sigma, \text{ defined by } (y, x) \to x,$$
and
$$Y(j_\sigma): Y(P_\sigma) \to Y, \text{ defined by } (y, x) \to y.$$
We will identify the fiber $Y(\sigma) = \pi^{-1}(\sigma)$ with the fiber of $\pi_\sigma$ above the vertex of $P_\sigma$ mapped by $i_\sigma$ on $\sigma$. For an edge $a = (\sigma < \tau)$, the inclusion map $i_a : P_\sigma \to P_\tau$ lifts to a map
$$Y(i_a) : Y(P_\sigma) \to Y(P_\tau), \text{ defined by }(x,y) \to (i_a(x), y).$$
Suppose that a section $s$ of $\pi$ over the 1-skeleton $|\P|^1$ of $|\P|$ is given. In particular each fiber $Y(\sigma)$ has a base point $s(\sigma)$. For each $\sigma \in \P$, this induces a section $s_\sigma$ of $\pi_\sigma$ over the 1-skeleton of $P_\sigma$.

\begin{dfn}
A \emph{poset of spaces} over $\P$ is a space $Y$ together with a projection map $\pi: Y \to |\P|$ and a section $s : |\P|^1 \to Y$ over the 1-skeleton of $|\P|$ so that
\begin{itemize}
\item For each $\sigma \in \P$, $Y(\sigma)$ is path-connected and there is a retraction $r_\sigma: Y(P_\sigma) \to Y(\sigma) \subset Y(P_\sigma)$ (the identification is indicated above) which is homotopic to the identity relatively to $Y(\sigma)$ and which is compatible with $s$, i.e., $r_\sigma s_\sigma(x) = s(x)$ for all $x$ in the 1-skeleton of $P_\sigma$.
\item Given an edge $a$ from $\sigma$ to $\tau$, let $\theta: Y(\sigma) \to Y(\tau)$, where $\theta (y) := r_\tau Y(i_a)(y)$. If $G_\sigma = \pi_1(Y(\sigma), s(\sigma))$ and if $a$ is an edge from $\sigma \to \tau$, then the homomorphism $\phi_a : G_\sigma \to G_\tau$ induced by $\theta$, is injective.
\end{itemize}
\end{dfn}

The groups $G_\sigma$ and the homomorphisms $\phi_a : G_{i(a)} \to G_{t(a)}$ determine a poset of groups $\G(\P)$ \cite{Haefliger}. We say that $Y \to |\P|$ is a \emph{realization} of $\G(\P)$. 

\begin{dfn}
Suppose $Y$ is a poset of CW complexes over $\P$ with associated poset of groups $\G(\P)$. Then $Y$ is an \emph{realization with aspherical complexes} of $\G(\P)$ if each $Y(\sigma)$ has the homotopy type of the classifying space $BG_\sigma$.
\end{dfn}

\begin{claim}
The classifying space $BG(\P)$ of the poset of groups $\G(\P)$ is a realization with aspherical complexes of $\G(\P)$.
\end{claim}


\textbf{Explicit description.} Let $\G(\P)$ be a poset of groups over a poset $\P$. For each $\sigma \in \P$, let $D_\sigma$ denote the dual cone $|\P_{\ge \sigma}|$ and choose a model for $BG_\sigma$. By using the mapping cylinder, we may assume that if $\sigma < \tau$, then $BG_\sigma$ is a subcomplex of $BG_\tau$. Thus, when $\sigma < \tau$, we have
$$BG_\sigma < BG_\tau \text{ and } D_\tau < D_\sigma$$

To form the realization of the complex of groups $\G(\P)$, we glue together the spaces $BG_\sigma \times D_\sigma$ in the following fashion: whenever $\sigma < \tau$ glue  $BG_\sigma \times D_\sigma$ to $BG_\tau \times D_\tau$ by identifying them along the common subspace $BG_\sigma \times D_\tau$. 

\begin{claim}
The resulting CW-complex $Y$ is a realization with aspherical complexes of $\G(\P)$.
\end{claim}
\begin{proof}
First of all, we need to prove that the resulting CW-complex is a poset of spaces over $\P$. There is a projection $\pi : Y \to |\P|$ induced by the projections of $BG_\sigma \times D_\sigma$ on the second factor. The gluing pattern on the second factors of each component of $Y$ is the same as the one for making $|\P|$. Hence, the projection map is continuous.

For each $\sigma \in \P$, choose a base point $s(\sigma)$ in each Eilenberg-MacLane space $BG_\sigma$; if $\sigma < \tau$ then $s(\tau) = s(\sigma)$. We then can find a section $s: |\P|^1 \to Y$ over the 1-skeleton of $|\P|$. Put $Y(\sigma) := \pi^{-1}(\sigma) = BG_\sigma$- a path-connected space. $Y(P_\sigma)$ is the graph of the resection of $\pi$ to $\pi^{-1} (P_\sigma)$ where $P_\sigma$ is $|\P_{\le \sigma}|$. $P_\sigma$ is a cone with the cone point $\sigma$. For each $\nu < \sigma$, $BG_\nu = \pi^{-1}(\nu)$ is a subcomplex of $BG_\sigma$. Thus, there is a retraction $Y(P_\sigma) \to Y(\sigma)$, which maps each space $BG_\nu$ for each $\nu < \sigma$ to $BG_\sigma$ by injection map. And hence, $Y$ is a realization of $G(\P)$.
\end{proof}

The following lemma is from \cite{Haefliger}.
\begin{lemma}\label{lemma:equivalence}
Any realization with aspherical complexes is homotopy equivalent to the classifying space $BG(\P)$ of the poset of groups $\G(\P))$.
\end{lemma}

\begin{lemma}\label{lemma:flagcomplex}
(\cite{CharneyDavis1}) Suppose that $L$ is a simplicial complex and $\G(\SL)$ is a developable poset of groups over the poset $\SL$. Let $G$ denote the colimit of $\G(\SL)$ and let $Y \to |\SL|$ be a realization with aspherical complexes. If $L$ is a flag complex, then $U(G, |\SL|)$ is contractible and $Y$ is aspherical.
\end{lemma}

\begin{example}
Consider the poset of groups over the poset of simplices of a graph $L$ consisting of 4 vertices and 3 edges (see Figure \ref{figure:PosetofGroups}).

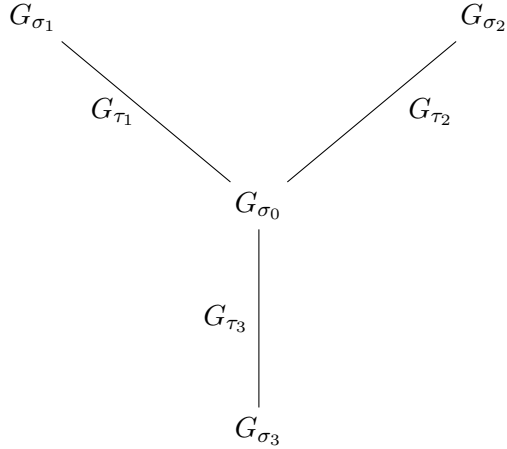
\begin{figure}
\begin{center}

\begin{tikzpicture}

\begin{scope}
    \node (A) at (3,0) {$G_{\sigma_0}$};
    \node (B) at (0,2.5) {$G_{\sigma_1}$};
    \node (C) at (6, 2.5) {$G_{\sigma_2}$};
    \node (D) at (3, -3) {$G_{\sigma_3}$};
\end{scope}

\draw (A) -- (B) node [midway, left] {$G_{\tau_1}$};
\draw (A) -- (C) node [midway, right=10pt] {$G_{\tau_2}$};
\draw (A) -- (D) node [midway, left] {$G_{\tau_3}$};
\end{tikzpicture}

\caption{A poset of groups}
\label{figure:PosetofGroups}
\end{center}
\end{figure}

The dual cones of the edges are points, the dual cones of vertices with local groups $G_{\sigma_1}$, $G_{\sigma_2}$ and $G_{\sigma_3}$ are segments; while the dual cone of the vertex with local group $G_{\sigma_0}$ is a cone on 3 points, which has the form of letter Y.

The gluing pattern is described in Figure \ref{figure:Gluingpattern}.

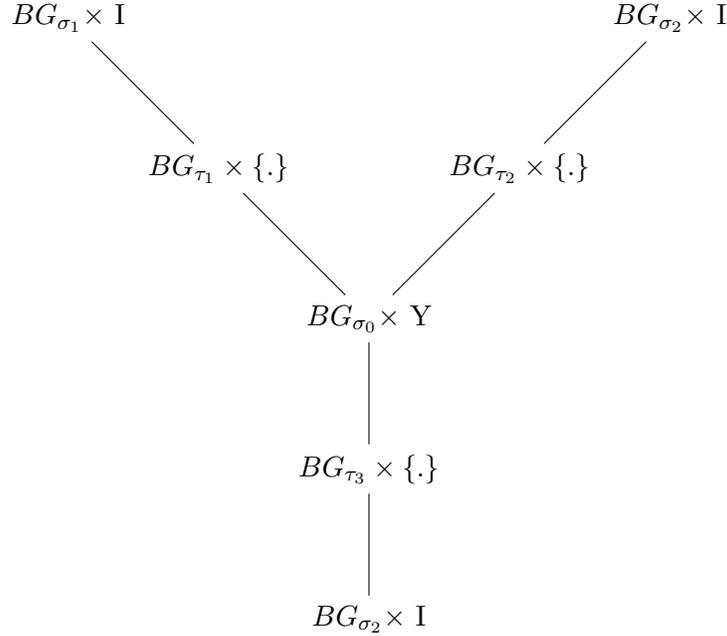
\begin{figure}
\begin{center}
\begin{tikzpicture}
\begin{scope}[every node/.style ={fill = white}]
\node (1) at (4,4) {$BG_{\sigma_0} \times$ Y};
\node (2) at (2,6) {$BG_{\tau_1} \times \{.\}$ };
\node (3) at (0,8) {$BG_{\sigma_1} \times $ I};
\node (4) at (6,6) {$BG_{\tau_2} \times \{.\}$};
\node (5) at (8,8) {$BG_{\sigma_2} \times $ I};
\node (6) at (4,2) {$BG_{\tau_3} \times \{.\}$};
\node (7) at (4,0) {$BG_{\sigma_2} \times $ I};
\end{scope}
\draw (1) -- (2) node[] {};
\draw (3) -- (2) node[] {};
\draw (1) -- (4) node[] {};
\draw (4) -- (5) node[] {};
\draw (1) -- (6) node[] {};
\draw (6) -- (7) node[] {};
\end{tikzpicture}

\caption{Gluing pattern}
\label{figure:Gluingpattern}
\end{center}
\end{figure}

\end{example}

\section{Spherical Artin groups}\label{section:Spherical_Artin_Groups}

\subsection{Artin group as a poset of groups}
We recall in this subsection a description of a poset of Artin groups given in \cite{CharneyDavis}. Associated with a simplicial complex $L$ with edge-labels by integers $\ge 2$ are a Coxeter group $W_L$ and an Artin group $A_L$. For each simplex $\sigma$ in $L$, the Coxeter group generated by vertices of $\sigma$ with corresponding relations is a subgroup of $W_L$, and for simplices $\sigma \subset \tau$, $W_\sigma$ is a subgroup of $W_\tau$ . The same statement is also true for $A_L$. This was proved by Deligne \cite{Deligne} for Artin groups of finite type (spherical Artin groups) and by van der Lek \cite{vanderLek} for general Artin groups.

Now consider two posets $\PL$ and $\SL$ where $\PL$ is the poset of simplices in $L$ and $\SL$ is $\PL$ union with an empty simplex; the partial order in both posets is defined by inclusion. We can define two posets of groups: $\G(\PL)$ and $\G(\SL)$. The first poset $\G(\PL)$ is defined as follows: for each $\sigma$ in $\PL$ the local group $G_\sigma$ is the Artin group $A_\sigma$. For $\sigma < \tau$ there is an injection $A_\sigma \to A_\tau$ and the maps are compatible as in the definition of poset of groups. For $\G(\SL)$, the local groups for $\sigma \ne \emptyset$ are the same as the local groups in $\G(\PL)$, the local group at $\emptyset$ is the trivial group.

It is easy to see that the colimit group of both posets of groups $\G(\PL)$ and $\G(\SL)$ is the Artin group $A_L$. But the fundamental groups are different if $L$ is not simply connected. For any $L$ the fundamental group of $\G(\SL)$ is $A_L$ because $|\SL|$ is simply connected (it is a cone on the barycentric subdivision of $L$).

\subsection{The $K(\pi, 1)$-Conjecture}\label{section:kpi1}
A central question of the studying of Artin groups is the $K(\pi, 1)$-Conjecture, which was firstly stated by Brieskorn \cite{Brieskorn} for some spherical Artin groups. The conjecture was then extended to a more general setting, which according to van der Lek \cite{vanderLek}, is due to Arnold, Pham and Thom. The $K(\pi, 1)$-Conjecture states that  a certain complex (defined from complement of hyperplane arrangement) is a Eilenberg-MacLane space for the Artin group.

Charney and Davis \cite{CharneyDavis}  describe a poset of groups over the poset $\P = \SL$ of the nerve $L$ of an Artin group $A_L$. As is pointed out in \cite{CharneyDavis}, this poset of groups is developable and the fundamental group (which is also the colimit group) is the Artin group $A_L$.  It is proved in \cite{CharneyDavis} that the $K(\pi, 1)$-Conjecture for Artin groups is equivalent to the following conjecture.
\begin{conjecture}\label{conj:kpi1}
The universal cover, $U(A_L, |\SL|)$, is contractible.
\end{conjecture}
$U(A_L, |\SL|)$ is the basic construction which we described in Section  \ref{section:Poset_of_groups}. One of the main results of \cite{CharneyDavis} is a proof of Conjecture \ref{conj:kpi1} when $L$ is a flag complex (recall that a flag complex is a complex which can be determined by its 1-skeleton, i.e., a collection of $k$ vertices of the complex forms a simplex if and only if every pair of those vertices is connected by an edge).

\bigskip

\subsection{A $K(\pi, 1)$ manifold model for spherical Artin groups}\label{section:kpi1forSphericalAG}
We will study a manifold for spherical Artin groups. This manifold comes from complement of hyperplane arrangement in a sphere. The manifold and its properties may have known to experts, but we could not find any references, thus, we write up the details here. 

\textbf{A $K(\pi, 1)$ manifold model and its boundary.}
We have a natural map from an Artin group $A_L$ to a Coxter group $W_L$, by sending generators of $A_L$ to corresponding generators of $W_L$. The kernel of this map is a normal subgroup of $A_L$ and it is called the \emph{ pure Artin group}:
$$0 \to PA_L \to A_L \to W_L \to 0.$$

Suppose for now that $L = \sigma$ is a simplex ($A_L$ is spherical Artin group), thus $W_L = W_\sigma$ is a finite Coxeter group. This Coxeter group $W_\sigma$ acts as real linear reflection group on $\mathbb{R}^{d+1}$, where $d$ is dimension of $\sigma$. The reflections in $W_\sigma$ are the orthogonal reflections across the hyperplanes $H$. Complexification gives a linear action on $\mathbb{C}^{d+1} = \mathbb{R}^{d+1} \otimes \mathbb{C}$. Denote by $\mathcal{A}$ the (finite) set of linear hyperplanes, $\mathcal{A} = \{ H \otimes \mathbb{C} \}$. A \emph{subspace} $G$ of $\mathcal{A}$ is an intersection of hyperplanes in $\mathcal{A}$. It is a fixed point set of some subgroup of the Coxeter group. Denote by $L(\mathcal{A})$ the poset of all subspaces of $\mathcal{A}$. The union of all subspaces of $\mathcal{A}$ is the set of points with nontrivial isotropy group. This set can be viewed as stratified space with strata indexed by the poset $\P = \P (\sigma)$. Define $M(\mathcal{A})$, the \emph{complement} of $\mathcal{A}$, by 
$$M(\mathcal{A}) = \mathbb{C}^{d+1} - \bigcup\limits_{H \in \mathcal{A}} H. $$ 
Also denote $\SA = \MA/\mathbb{R}_+ \cong \MA \cap S^{2d+1}$. Since
$$\MA \cong \SA \times \mathbb{R}_+,$$
we see that $\SA$ is homotopy equivalent to $\MA$. Deligne \cite{Deligne} proved that $\MA$ is homotopy equivalent to $K(PA_\sigma, 1)$, and therefore, so is $\SA$. Our goal is to define a bordification of $\MA$ by a manifold with faces $\M$, with its faces indexed by $L(\mathcal{A})$, as well as a compactification $\BS$ of $\SA$, so that the interior of $M$ is $\MA$ and the interior of $\BS$ is $\SA$.

\medskip

Given a subspace $G \in L(\mathcal{A})$, define a subarrangement $\A_G$ of $\A$, and a hyperplane arrangement $\A^G$ in $G$:
$$\A_G := \{ H \in \A | G \le H \}$$
$$\A^G := \{ G \cap H | H \in \A - \A_G \}.$$
Denote by $G^\perp$ the perpendicular complement of $G$ in $\mathbb{C}^{d+1}$, i.e., $\mathbb{C}^{d+1} = G \oplus G^\perp$, then $\A_G$ is naturally identified with a hyperplane arrangement in $G^\perp$, namely,
$$\A_{G^\perp} = \{ H \cap G^\perp | H \in \A_G \}$$

We will blowup $\mathbb{C}^{d+1}$ along all subspaces of $\A$ starting with the subspaces of smallest dimension. For a discussion about blowing up, see \cite{Davis} and \cite{Janich}. This will result in a sequence of manifolds with faces, $\mathbb{C}^{d+1} = M(-1) < M(0) < \ldots < M(d)$, where $M(i)$ is obtained from $M(i-1)$ by blowing up the subspaces of dimension $2i$. The first step of the blowing up procedure is blowing up the smallest subspace $F = \{ 0 \}$. And the result is 
$$M(0) = S^{2d +1} \times [0, \infty).$$
The face $\d_FM(0)$ corresponding to $F = \{ 0\}$ is $S^{2d +1} \times 0$. Similarly, each subspace $G$ of $\A$ is blown up to a manifold with boundary $G(0)$ with $\d_FG(0) = S^{2d +1} \cap G$. Suppose, by induction, that $M(i -1)$ has been defined and for each subspace $G$ of dimension $\ge 2i$, submanifold with faces $G(i-1)$ also has been defined. Let $G$ be a subspace of dimension $2i$. Blowup $M(i-1)$ along $G(i-1)$ to get $M(i-1) \odot G(i-1)$. In the process, we introduce a new face indexed by $G$, $\d_GM(i-1) \odot G(i-1)$. Similarly, we can blow up all other $2i$-dimensional subspaces to obtain $M(i)$. Then we can define for each subspace $G$ of dimension $\ge 2(i+1)$: $G(i) := M(i) \cap G$. We continue this process to $i = d$ and obtain $\M := M(d)$. Finally, define
$$\BS := \d_0 M(d),$$
where $\d_0()$ denotes the blowup corresponding to a minimum face $F = \{ 0 \}$.
\medskip

Next, we will study closely the boundary of $\M$ and $\BS$. Let $\rho < \sigma$ be a simplex in $\sigma$, there is the Coxeter group $W_\rho < W_\sigma$. Let $G$ be the fixed subspace of $W_\rho$. When we blow up all the subspaces contained in $G$ (which are the elements of $A^G$), we get a submanifold $\bar{M}(A^G)$ in an appropriate blowup $N$ of $\mathbb{C}^{d+1}$. The normal bundle of $\bar{M}(A^G)$ in $N$ is trivial with fiber $\mathbb{C}^{c(\rho)}$, where $c(\rho)$ means the codimension of $\rho$ in $\sigma$. Consider $G^\perp$, the hyperplane arrangement in $G^\perp$ is the reflection arrangement of $W_\rho$. If we continue to blow up $N$ along subspaces for the reflection arrangement of $W_\rho$, we obtain a part of $\M$, which is 
$$ \bar{M}(\A_{G^\perp}) \times \bar{M}(\A^G).$$
A part of the boundary of this submanifold is 
$$\d_G\M = \bar{S}(\A_{G^\perp}) \times \bar{M}(\A^G)$$
and a corresponding face in $\BS$
$$\d_G\BS = \bar{S}(\A_{G^\perp}) \times \bar{S}(\A^G) $$

The Coxeter group $W_\sigma$ acts on $\M$ and $\BS$. Denote the manifold $\BS / W_\sigma$ by $M_\sigma$. It follows that part of the boundary of $M_\sigma$ is $(\bar{S}(\A_{G^\perp})/W_\rho)  \times \bar{S}(\A^G)$, where $ \bar{S}(\A_{G^\perp})/W_\rho = M_\rho$. We summarize the above argument into the following proposition.
\begin{proposition}\label{prop:localmanifold}
Suppose $A_\sigma$ is a spherical Artin group with the nerve $\sigma$ of dimension $d$. Then there is a $(2d+1)-$dimensional manifold with boundary $M_{\sigma}$ with the following properties.
\begin{enumerate}
\item $M_{\sigma}$ is homotopy equivalent to $BA_{\sigma}$ (i.e., $M_\sigma$ is aspherical and its fundamental group is isomorphic to $A_\sigma$).
\item if $\rho < \sigma$ is a proper face, then $M_\rho \subset \d M_\sigma$. Moreover, the normal bundle of $M_\rho$ in $\d M_\sigma$ is trivial.
\end{enumerate}
\end{proposition}

The following lemma is needed for our construction in Section \ref{section:Manifold_construction}.
\begin{lemma}\label{lemma:disjointsubmanifolds}
The boundary of $M_\sigma$ contains a collection of disjoint manifolds $M_\rho$ for every $\rho < \sigma$.
\end{lemma}
\begin{proof}
From Proposition \ref{prop:localmanifold} we know that $M_\rho$ is contained in the boundary of $M_\sigma$. We will prove that we can choose $M_\rho$ such that it has no intersection with other submanifolds.

Since the statement of Proposition \ref{prop:localmanifold} is true for any simplex $\sigma$, it follows that if $\rho$ is simplex of codimension 1 of $\sigma$, then the boundary of $M_\rho$ contains submanifolds $M_\pi$ for any $\pi < \rho < \sigma$. Thus, it is enough to choose a disjoint collection of $M_\rho$ of $M_\sigma$ where $\rho$ is simplex of codimension 1 (one manifold $M_\rho$ can have multiple disjoint copies, one for each $\pi < \rho$, in the collection).

Let $\rho_1$ and $\rho_2$ be two simplexes of $\sigma$ of codimenion 1. We will study the manifold $M_{\rho_1}$ and $M_{\rho_2}$ in the hyperplane complement and in its quotient by the Coxeter group $W_\sigma$. Let $G_1$ and $G_2$ be the fixed subspaces of $W_{\rho_1}$ and $W_{\rho_2}$. $G_1$ and $G_2$ are planes in $\mathbb{C}^{d + 1}$ (2-dimensional vector spaces). They intersect at the origin (the fixed subspace of $W_\sigma$). After blowing up the origin, these two subspaces have no intersection. Thus, the two manifolds $\bar{S}(\A_{G_1})$ and $\bar{S}(\A_{G_2})$ are totally disjoint because they sit in the sphere bundle of each $G_1$ and $G_2$. If the group $W_\sigma$ does not map $G_1$ to $G_2$, then these 2 manifolds remain disjoint in the quotient of the hyperplane complement by the group $W_\sigma$. If the group $W_\sigma$ does map $G_1$ to $G_2$, then the two manifold are the same in the quotient. But since the normal bundle of this manifold in the boundary of $M_\sigma$ is trivial, we can take as many copies of $M_{\rho_1}$ (or $M_{\rho_2}$) as we want. This proves the lemma.
\end{proof}

\bigskip

\subsection{Geometric and action dimensions}
Let $\Gamma$ be a discrete torsion-free group. We have the following notions of dimensions for $\Gamma$.
\begin{dfn} \label{dfn:gdim}
The {\em geometric dimension}, gdim$(\Gamma)$, of $\Gamma$ is the minimal $n$ such that $\Gamma$ admits a properly discontinuous action on a contractible $n$-complex, in other words, $\Gamma$ has a finite $n$-dimensional classifying space $B\Gamma$.
\end{dfn}
\begin{dfn} \label{dfn:actdim}
The {\em action dimension}, actdim$(\Gamma)$, is the smallest integer $n$ such that $\Gamma$ admits a properly discontinuous action on a contractible $n$-manifold. If no such $n$ exists, then actdim$(\Gamma) = \infty$.
\end{dfn}
Notice that if $\Gamma$ admits a proper action on a contractible manifold $M$ and if $\Gamma$ is torsion-free, then the action is free and $M/\Gamma$ is a finite dimensional model for $B\Gamma$. A theorem of Stallings \cite{Stallings} says that every $n$-complex is homotopy equivalent to a complex that embeds in $\mathbb{R}^{2n}$. It follows that for a torsion free group $\Gamma$ we have 
$$\text{actdim}(\Gamma) \le  2 \text{gdim}(\Gamma),$$
We could see that by immersing an $n$-complex (a $K(\Gamma, 1)$ model) into $\mathbb{R}^{2n}$ and taking its regular neighborhood.

For a RAAG,  its action dimension is studied in \cite{ADOS}. In particular, the authors prove:
\begin{theorem}\label{theorem:RAAGs}
Suppose $L$ is a $k$-dimensional flag complex and $A_L$ is the RAAG. Then, if $H_k(L;\mathbb{Z}/2) \ne 0$, then actdim$(A_L) = 2k + 2$. If $H_k(L;\mathbb{Z}/2) = 0$ and $k \ne 2$, then actdim$(A_L) \le 2k + 1$.
\end{theorem}

Let us discuss some of the concepts and ideas of the proof of the above theorem; for more details, see \cite{ADOS} . To prove the theorem, the authors introduce the so-called \emph{octahedralization}, $OL$ of a simplicial complex $L$. If $V$ is the set of vertices of $L$, then the set of vertices of $OL$ is $V \times \{ \pm 1 \}$; and if a collection of vertices $\{ v_0, \ldots, v_k\}$ forms a simplex in $L$, then any subset of the form $\{ (v_0, \epsilon_0), \ldots, (v_k, \epsilon_k) \}$ of $V \times \{ \pm 1 \}$ spans a $k$-simplex. The octahedralization $OL$ of the nerve $L$ of a RAAG $A_L$ sits inside the boundary $\partial_\infty A_L$ of $A_L$ and plays a role of obstructor complex (\cite{BKK}) for the RAAG. There are two dimensions of $OL$: one is the \emph{embedding dimension}, which is the minimum dimension $m$ in which $OL$ can embed piecewise linearly in $S^m$. The other is the \emph{van Kampen dimension} - the maximum $m$ such that some particular cohomology class of the configuration space of $OL$ is nonzero. The action dimension of $A_L$ is related to the embedding dimension and the van Kampen dimension of $OL$ in the right-angled case by the following inequalities:
$$\text{embdim} OL + 1 \ge \text{actdim} A_L \ge \text{vkdim} OL + 2.$$
The main result of \cite{ADOS} concerns vkdim$OL$ for a flag complex $L$ ( recall that for a RAAG, the nerve $L$ is a flag complex). In the case when the top homology $H_k$ of $L$ with $\mathbb{Z}/2$-coefficients is nonzero, the van Kampen dimension of $OL$ is $2k$; consequently, the action dimension of $A_L$ is $2k + 2$.


\subsection{Geometric dimension of spherical Artin groups}
For spherical Artin group $A_\sigma$, we have the following theorem about its geometric dimension.
\begin{proposition}
The geometric dimension of a spherical Artin group $A_\sigma$ is dim($\sigma$) + 1.
\end{proposition}
\begin{proof}
It was proven in\cite{Deligne} that spherical Artin groups admit finite dimensional classifying space. Alternatively, we can see that spherical Artin groups has the nerve $L$ a simplex, which is obviously a flag complex. Then, by the result of \cite{CharneyDavis}, these groups satisfy the $K(\pi,1)$-Conjecture. Charney and Davis proved that there is a finite Eilenberg-MacLane space of dimension $\sigma$ + 1 for the Artin group. Thus, gdim($A_\sigma$) $\le$ dim($\sigma$) + 1.

On the other hand, the cohomology of $A_\sigma$ with $\mathbb{Z}$-coefficients is nonzero in the top dimension, which is equal to dim($\sigma$) + 1 (\cite{Davis}). Since the cohomology dimension of a group is a lower bound for geometric dimension, the theorem follows.
\end{proof}


\subsection{Action dimension of irreducible spherical Artin groups}\label{subs:actdim_spherical}
The finite (irreducible) Coxeter groups were classified by Coxeter in 1935, see \cite{Coxeter}. They consist of three one-parameter families of increasing rank $A_n, B_n, D_n$, one one-parameter family of dimension two, $I_2(p)$ and six exceptional groups $H_3, H_4, F_4, E_6, E_7$ and $E_8$. Will denote the Artin groups of type $A_n, B_n, \ldots$ by $A(A_n), A(B_n), \ldots$.


Since the geometric dimension of a spherical Artin group $A_\sigma$ is dim($\sigma$) + 1, the action dimension of $A_\sigma$ is less than or equal to 2dim($\sigma$) + 2. In fact, the action dimension of irreducible spherical Artin group is 2dim($\sigma$) + 1.


\begin{theorem}
Let $\Gamma$ be the Artin group corresponding to Coxeter groups of type $A_n, B_n$ or $D_n$. Then the action dimension of $\Gamma$ is $2n - 1$.
\end{theorem}
\begin{proof}
The number $n$ in the notation of $A_n, B_n$ or $D_n$ is the number of generators of each Artin group, which is also the geometric dimension of the (spherical) Artin group. For the Artin group corresponding to the Coxeter groups of type $A_n, B_n$ and  $D_n$ we have an aspherical manifold model (from the associated hyperplane arrangement, see Section \ref{section:kpi1forSphericalAG}) of dimension $2n - 1$. Thus,  actdim($\Gamma$)  $\le 2n - 1$. We need to prove the opposite direction, i.e., actdim($\Gamma$)  $\ge 2n - 1$.

\textbf{Type $A_n$}. In \cite{BKK} the authors defined a notion of obstructor dimension for a group, obdim($\Gamma$). It was pointed out in that paper that the obstructor dimension is a lower bound for action dimension, i.e., for any group $\Gamma$, we have actdim($\Gamma$) $\ge$ obdim($\Gamma$). The authors also proved that obdim$(\Gamma)$ = $2n - 1$ for $\Gamma$ is the braid group on (n+1) strands. It follows that actdim$(\Gamma) \ge 2n - 1$.

\textbf{Type $B_n$}. There is a decomposition of the Artin group of type $B_n$, $A(B_n)$, as a semidirect product of free group and braid group on $n$ strands $\mathcal{B}_n$ (\cite{CrispParis}, proposition 2.1):
$$A(B_n) \cong F_n \rtimes \mathcal{B}_n.$$
Corollary 27 from \cite{BKK} states that if groups $H$ and $Q$ are finitely generated, $H$ is weakly convex and if $G = H \rtimes Q$, then obdim$G$ $\ge$ obdim$H$ + obdim$Q$. Examples of weakly convex groups are hyperbolic, CAT(0) or semi hyperbolic groups. Since both $F_n$ and $\mathcal{B}_n$ are finitely generated and $F_n$ is hyperbolic, we have
$$\text{obdim}(A(B_n)) \ge \text{obdim}(F_n) + \text{obdim}(\mathcal{B}_n).$$ 
The obstructor dimension of a free group is 2, obdim($\mathcal{B}_n$) = $2(n-1) - 1 = 2n - 3$. Thus,
$$\text{obdim}(A(B_n)) \ge 2 + 2n - 3 = 2n -1.$$
So,
$$ \text{actdim}(A(B_n)) \ge 2n - 1.$$

\textbf{Type $D_n$}. The Artin group of type $D_n$ can also be decomposed as a semidirect product of free group and braid group on $n$ strands for $n \ge 4$ (\cite{CrispParis}, proposition 2.3)
$$A(D_n) \cong F_{n-1} \rtimes \mathcal{B}_n.$$
By the same argument as above, we have actdim$(A(D_n)) \ge 2n - 1$.
\end{proof}

For groups $A(I_2(p))$, we have actdim($A(I_2(p))$) = 3. This is from \cite{Gordon} based on previous work \cite{Droms}. In \cite{Gordon}, the 3-manifold Artin groups with 1-dimensional nerves were classified. The 3-manifold Artin groups are the ones which have the nerve is a tree or a triangle with all labels 2.

\begin{theorem}
We have the following equalities:
\begin{itemize}
\item actdim($A(H_3)$) = 5,
\item actdim($A(F_4)$) = 7,
\item actdim($A(H_4)$) = 7,
\item actdim($A(E_6)$) = 11,
\item actdim($A(E_7)$) = 13,
\item actdim($A(E_8)$) = 15.
\end{itemize}
\end{theorem}

\begin{proof}
In the upcoming work by Davis and Huang \cite{DH}, an analog of $OL$ as in the RAAG case for general Artin groups is defined. They do so by studying the so called \emph{"standard abelian subgroups"} of spherical Artin groups. They are the maximal abelian subgroups of each (standard) parabolic subgroup of a particular Artin group. This collection of abelian subgroups can be encoded by a complex, which is a subdivision of the nerve $L$. For irreducible spherical Artin groups. For irreducible spherical Artin groups, this complex contain the barycenter of the simplex $L$.
Let $L = \sigma$ be the simplex of correct dimension $n$ for each Artin group of type from the list $\{  H_3, F_4, H_4, E_6, E_7, E_8 \}$ ( $n = 2$ for $A(H_3)$, $n = 3$ for $A(F_4)$ and so on).  The barycenter, denoted by $v$, of $L$ has the link is a triangulation of the sphere $S^{n-1}$. The link of $v$,  Lk$(v)$, which is the boundary of the barycentric subdivision of $L$, is a flag complex. And thus, by the main result of \cite{ADOS}, the van Kampen dimension of the octahedralization of Lk$(v)$ is $2(n-1)$. Now, the barycentric subdivision of $L$ is the join of $v$ and its link: $v * \text{Lk}(v)$; by the definition of the octahedralization of a complex, we have the octahedralization of $v * \text{Lk}(v)$ is the suspension of the octahedralization of the link of $v$. Thus, the van Kampen dimension of the octahedralization of $v * \text{Lk}(v)$ is $2(n -1) + 1 = 2n - 1$. Since the action dimension of $A_L$ is greater or equal to the van Kampen dimension of any of its obstructor complexes plus 2, it follows that actdim$(A_L) \ge 2n - 1 + 2 = 2n + 1$.Since we already proved the other direction of the inequality, the theorem follows. 
\end{proof}
\begin{remark}
The above proof works for all irreducible spherical Artin groups. 
\end{remark}

\section{Contractible complexes}\label{sec:Contractibility}
This section deals with the question of when a simplicial complex $L$ can be embedded into a contractible complex of the same dimension; and from that, we will define a new poset of groups on the poset of a contractible complex. 

\subsection{Embedding into contractible complexes}
The following lemma is standard. See \cite{Fenn}, for example, for references.
\begin{lemma} \label{lemma:EmbedinContractible1}
Suppose $L$ is a (connected) finite simplicial complex of dimension $d$, $d \ne 2$. Then $L$ can be embedded in a contractible complex $C$ of the same dimension if and only if $H^d(L, \mathbb{Z}) = 0$. 
\end{lemma}
We treat the case $d = 2$ separately.
\begin{lemma}\label{lemma:EmbedinContractible2}
Suppose $L$ is 2-dimensional simplicial complex. If $H^2(L, \mathbb{Z}) = 0$ and $\pi_1(L)$ is normally generated by $r$ elements where $r = rk H_1(L, \mathbb{Z})$, then $L$ can be embedded in a contractible complex $C$ of the same dimension.
\end{lemma}

Note that the complexes in Lemmas \ref{lemma:EmbedinContractible1} and \ref{lemma:EmbedinContractible2} can be made to be simplicial. This is because of the Simplicial Approximation Theorem \cite{Hatcher}.
\begin{theorem}
If $K$ is a finite simplicial complex and $L$ is an arbitrary simplicial complex, then any map $f: K \to L$ is homotopic to a map that is simplicial with respect to some iterated barycentric subdivision of $K$.
\end{theorem}

\subsection{A new poset of groups}\label{section:newcomplexofgroups}
Let $L$ be the nerve of an Artin group $A_L$. We have a poset of group $\G(\PL)$ as described in Section \ref{section:Poset_of_groups}. The colimit group of this poset of groups is the Artin group. Let $C$ be a contractible simplicial complex, which contains $L$ as subcomplex. We will define a new complex of groups $\G(\P(C))$, which has the same colimit groups $A_L$.

\medskip

\textbf{A new poset}. For each $\sigma \in \P(C)$, if $\sigma$ is also in $\PL$, define the local group at $\sigma$ in $\G(\P(C))$ is the same as the local group at $\sigma$ in $\G(\PL)$. If $\sigma$ is not in $\PL$, let $\rho$ be the maximal simplex in $\PL$ such that $\rho < \sigma$, then let the local group at $\sigma$ be the local group $G_\rho$ of $\G(\PL)$. If there is no such $\rho$, then let the local group at $\sigma$ be the trivial group. Obviously, the two complexes of groups have the same colimit group, which is the Artin group. We will use this complex of groups in the construction in Section \ref{section:Manifold_construction}.

\medskip

\textbf{Aspherical manifolds for local groups}. For each $\sigma$ of dimension $i$ in $\PL$, there is a $K(A_\sigma, 1)$ manifold of dimension $2i + 1$ (Proposition \ref{prop:localmanifold}). If $\sigma$ is not in $\PL$ and if the local group at $\sigma$ is trivial, define $M_\sigma = D^{2i + 1}$. If the local group at $\sigma$ is $A_\rho$ for $\rho$ is a maximal simplex in $\PL$ which is contained in $\sigma$, then let $M_\sigma = M_\rho \times D^{2i - 2h}$, where $h$ is the dimension of $\rho$. Thus, for each $\sigma$ in $\P(L_1)$, we have a manifold $M_\sigma$ of dimension $2i + 1$.

It is easy to see that the statements in Proposition \ref{prop:localmanifold} and lemma \ref{lemma:disjointsubmanifolds} remain valid for these new manifolds. 

\medskip

Now suppose that the Artin group $A_L$ satisfies the $K(\pi, 1)$-Conjecture. As we discussed in Subsection \ref{section:kpi1}, this is equivalent to saying that the basic construction $U(A_L, |\SL|)$ is contractible. It follows that, by Lemma \ref{lemma:UPandUS}, $U(G, |\P(C)|)$ is also contractible. 

We will consider the complex of groups $\G(\P(C))$ in the next section. But for convenience, we will use $L$ instead of $C$. Thus, we have a complex of group $\G(\PL)$ with $U(G, |\PL|)$ contractible and for each local group $G_\sigma$ of $\sigma \in \PL$, there is a $K(G_\sigma, 1)$ manifold model $M_\sigma$ of dimension $2i + 1$, such that the statements in Proposition \ref{prop:localmanifold} and Lemma \ref{lemma:disjointsubmanifolds} hold.

\section{Construction of the Manifold}\label{section:Manifold_construction}
\subsection{Thickening of a simplicial complex}
The notion of thickening of a simplicial complex is studied in a broad view in \cite{Mazur} and \cite{Wall}. Here, we will consider only a simple version of thickening, a thickening as a regular neighborhood of a complex of dimension $d$ in a manifold of dimension $2d + 1$ or $2d$. In the following subsections, we will denote a thickening of $L$ by Th$L$ or Th($L$, $M$) if we want to specify the manifold $M$.
 
\textbf{Basic case.}
It is well known that for any simplicial complex $L$ of dimension $d$, we have a simplicial embedding of $L$ as a subcomplex of some PL (piecewise linear) triangulation of $S^{2d+1}$. Thus, we have a thickening Th$L$ of $L$, which is a submanifold of $S^{2d+1}$. The restriction of this thickening to each $i-$simplex $\sigma$ is a disk of the form $\sigma \times D^{2d+1 - i}$, where $D^{2d+ 1 - i}$ is a disk of dimension $2d + 1 - i$. 

The \emph{thickening of the cone on} $L$ of dimension $d$, denoted by Th(Cone$L$), is the disk $D^{2d+2}$, together with the distinguished submanifold of $\d D^{2d+2} = S^{2d + 1}$, which is defined as the thickening of $L$ in $S^{2d + 1}$. (Note: the cone on $L$ has dimension $d+1$, thus it has an usual thickening in $S^{2d+3}$. But since it is a cone, we can find a thickening of lower dimension.).

\textbf{Thickened dual cones}
Let $L$ be a simplicial complex and $\sigma$ is a simplex in $L$. Recall that the dual cone, $D_\sigma$, of $\sigma$ is the cone on the barycentric subdivision of the link of $\sigma$ in $L$. In other words, let $\P = \PL$, then $D_\sigma = |\P_{\ge \sigma}|$ is its dual cone. Both $\sigma$ and its dual cone $D_\sigma$ are subcomplexes of the barycentric subdivision of $L$. If the codimension of $\sigma$ in $L$ is $k$, then 
\begin{itemize}
\item dim$Lk'_\sigma = k - 1$,
\item Th$Lk'_\sigma \subset S^{2k-1}$ is a submanifold of codimension 0,
\item and Th$D_\sigma$ is a disk of dimension $2k$.
\end{itemize}
 
Next suppose that $\tau > \sigma$ is a coface of $\sigma$ of codimension $l$ in $L$ (i.e., $\sigma$ is a subsimplex of $\tau$ and $l < k$). In $Lk'_\sigma$, $\tau$ is a vertex and the thickening of $\tau$ in the thickening of $Lk'_\sigma$, denoted by $T(\tau, Lk'_\sigma)$, is a disk of dimension $2k - 1$. We have
\begin{itemize}
\item $D_\tau \subset Lk'_\sigma$,
\item Th$D_\tau$ is a $2l$-disk, which can be embedded in the $(2k - 1)$-disk,  $T(\tau, Lk'_\sigma)$.
\end{itemize}
Thus, $T(\tau, Lk'_\sigma)$ splits as product of 2 disks
$$T(\tau, Lk'_\sigma) = D^{2(k-l) - 1} \times D^{2l}.$$

\subsection{The construction}
For each simplex $\sigma$ of dimension $i$ of the poset of groups $\G(\P(L))$, there is an aspherical manifold $M_\sigma$ of dimension $2i + 1$. The fundamental group of $M_\sigma$ is the local group at $\sigma$. Define
$$M(D_\sigma) := M_\sigma \times \text{Th} D_\sigma.$$
Since the thickening of the dual cone is a disk, $M(D_\sigma)$ contracts to $M_\sigma$. If the dimension of $L$ is $d$, then the codimension of $\sigma$ is $k = d - i$. Hence, the dimension of Th $D_\sigma$ is $2k = 2(d - i)$. And the dimension of $M(D_\sigma)$ is $2i + 1 + 2(d - i) = 2d + 1$. The dimension of $M(D_\sigma)$ does not depend on $\sigma$. Thus, we  define a manifold of dimension $2d + 1$, which contracts to $M_\sigma$ for any simplex $\sigma$ in $L$. We will use these manifolds to construct a manifold for the Artin group.

\textbf{Gluing}. In this construction we glue together the $M(D_\sigma)$ along parts of their boundaries in the same way the dual cones are glued together to get the barycentric subdivision of $L$. 

Let $\sigma < \tau$ be simplices in $L$ with dim($\sigma) = i$ and dim $(\tau) = j$. We have the boundaries of $M(D_\sigma) = M_\sigma \times \text{Th} D_\sigma$ and $M(D_\tau) = M_\tau \times \text{Th} D_\tau$ are
$$\d M(D_\sigma) = \d M_\sigma \times \text{Th} D_\sigma \cup M_\sigma \times \d  \text{Th} D_\sigma$$
$$\d M(D_\tau) = \d M_\tau \times \text{Th} D_\tau \cup M_\tau \times \d  \text{Th} D_\tau$$
Consider two parts of these boundaries: $M_\sigma \times \d  \text{Th} D_\sigma$ and $\d M_\tau \times \text{Th} D_\tau$, they have a common codimension 0 submanifold
$$M_\sigma \times D^{2(j-i) - 1} \times D^{2(d-j)}.$$
This follows from the following observations:
\begin{itemize}
\item T($\tau, Lk'_\sigma$) is a $2(d-i) - 1$-disk in $\d  \text{Th} D_\sigma$ and there is the splitting T($\tau, Lk'_\sigma$) = $D^{2(j - i) - 1} \times D^{2(d - j)}$,
\item \text{Th} $D_\tau = D^{2(d - j)}$, and  $M_\sigma \times D^{2(j-i) - 1} \subset \d M_\tau$.
\end{itemize}
 We can do this procedure for all simplices in $L$ and construct a manifold
 $$M(L) := ( \coprod_{\sigma \in \P} M(D_\sigma) ) / \sim.$$

 \begin{claim} 
 The result of the above construction is a manifold.
 \end{claim}
 \begin{proof}
 The only problem that prevents the construction result to be a manifold is that the same piece (or some part of it) of boundary of some local manifold is glued more than one time in $\d M_\tau \times \text{Th} D_\tau$. But this is not the case. By lemma \ref{lemma:disjointsubmanifolds}, we can choose a collection of disjoint sub-manifolds $M_\sigma$ for $M_\tau$. And for each such $M_\sigma$, choose a small disk around it to get $M_\sigma \times D^{2(j-i) - 1}$.

 \end{proof}
 
 Theorem \ref{theorem:maintheorem} follows from the following claim. 
\begin{claim}
The manifold $M(L)$ is aspherical.
\end{claim}
\begin{proof}
We know that $U(G, |\P(L)|)$ is contractible. Since $L$ is simply connected (it is contractible), by proposition \ref{prop:contractibility}, $BG(\P(L))$ is homotopy equivalent to the Eilenberg-MacLane space $BG$. And any realization is homotopy equivalent to $BG(\P(L))$ by Lemma \ref{lemma:equivalence}. Thus, the realization by aspherical manifolds $M$ is aspherical.
\end{proof} 

Since the manifold $M(L)$ we constructed is aspherical, its universal cover is contractible and it admits a proper action from the Artin group. This implies that the action dimension of the Artin group is less than or equal to the dimension of $M(L)$, which is $2n + 1$.

\textbf{Remark on Corollary \ref{cor:flagcomplex}}. When the nerve of the Artin group is a flag complex, we can prove Corollary \ref{cor:flagcomplex} directly using lemma \ref{lemma:flagcomplex} and the fact that we can embed $L$ into a contractible flag complex of the same dimension.

\textbf{Proof of Corollary \ref{cor:irreducible}}. From Section \ref{section:Spherical_Artin_Groups}, we know that the action dimension of an irreducible spherical Artin group with the nerve $\sigma$ a simplex of dimension $n$ is $2n + 1$. We also know that for any subgroup $G$ of a group $\Gamma$, we have
$$\text{actdim}(G) \le \text{actdim}(\Gamma).$$
Thus, if an Artin group $A_L$ has a subgroup which is an irreducible spherical Artin group $A_\sigma$ with dim($L$) = dim($\sigma$) = $n$, then actdim($A_L$) $\ge$ actdim($A_\sigma$) = $2n + 1$. Theorem \ref{theorem:maintheorem} implies that actdim($A_L$) = $2n + 1$.

\medskip

\bibliographystyle{alpha}

\bibliography{Actdim_Artin}

\end{document}